\documentclass[11pt, reqno]{amsart}

\RequirePackage[OT1]{fontenc}
\usepackage[colorlinks=true, urlcolor=blue, linkcolor=blue, citecolor=blue, pdftex]{hyperref}
\usepackage{csquotes}
\usepackage{capt-of}
\usepackage{amsmath}
\usepackage{float}
\usepackage[matrix,arrow,curve]{xy}

\usepackage{amsmath,amsfonts,verbatim,,afterpage,euscript,mathrsfs,amssymb}
\usepackage{amsfonts}
\usepackage{array}
\usepackage{amsmath,amssymb,amsthm}
\usepackage{psfrag}
\usepackage{dsfont}

\usepackage[utf8]{inputenc}
\usepackage[matrix,arrow,curve]{xy}
\usepackage{amsmath,amsfonts,verbatim,afterpage,euscript,mathrsfs,amssymb}
\usepackage{array}
\usepackage{amsthm}
\usepackage{dsfont}
\usepackage{hyperref}
\usepackage{enumerate}
\usepackage[english]{babel}
\usepackage{tikz}
\usetikzlibrary{shapes}
\usepackage{titlesec}
\usepackage{algorithmicx}
\usepackage{algpseudocode}
\usepackage{algorithm}

\titleformat{\section}[block]{\normalfont\bfseries\filcenter}{\itshape\thesection}{1em}{}
\titleformat{\subsection}[block]{\normalfont\bfseries}{\itshape\thesubsection}{0.9em}{}
\titleformat{\subsubsection}[block]{\normalfont\bfseries}{\itshape\thesubsubsection}{0.8em}{}
\titleformat{\caption}[block]{\normalfont}{\itshape}{0.8em}{}

\newtheorem{theorem}{Theorem}

\newtheorem{definition}{Definition}
\newtheorem{corollary}{Corollary}

\newtheorem{lemma}{Lemma}
\newtheorem{remark}{Remark}

\newcommand{\C}{{\rm Cov}}

\newcommand{\E}{\mathbb{E}}
\newcommand{\V}{\mathbb{V}}

\newcommand{\R}{\mathbb{R}}
\newcommand{\p}{\partial}
\newcommand{\z}{\mathbf{y}}
\newcommand{\Z}{\mathbf{Y}}
\newcommand{\y}{\mathbf{y}}
\newcommand{\Y}{\mathbf{Y}}
\newcommand{\Var}{{\rm Var}}

\begin{document}
\title[Parametric inference]{ Parametric inference of hidden discrete-time diffusion processes by deconvolution}

\author{Salima El Kolei}
\address{CREST, Ecole Nationale de la Statistique et de l'Analyse de l'Information (Ensai), Campus de Ker-Lann, rue Blaise Pascal, BP 37203, 35172 Bruz cedex, France. }
\email[Salima El Kolei]{salima.el\_kolei@ensai.fr}

\author{Florian Pelgrin}
\address{Ecole des Hautes Etudes Commerciales (EDHEC Business School), 24, avenue Gustave Delory CS 50411, 59057 Roubaix Cedex 1-France.}
\email[Florian Pelgrin]{florian.pelgrin@edhec.edu}


\begin{abstract}
We study a parametric approach for hidden discrete-time diffusion models based on contrast minimization and deconvolution. This approach leads to estimate a large class of stochastic models with nonlinear drift and nonlinear diffusion. It can be applied, for example, for ecological and financial state space models. \\
After proving consistency and asymptotic normality of the estimator, leading to asymptotic confidence intervals, we provide a thorough numerical study, which compares many classical methods used in practice (Monte Carlo Expectation Maximization Likelihood estimator and Bayesian estimators) to estimate stochastic volatility models. We prove that our estimator clearly outperforms the Maximum Likelihood Estimator in term of computing time, but also most of the other methods. We also show that this contrast method is the most stable and also does not need any tuning parameter.
\end{abstract}

\maketitle

\section{Introduction}
This paper is motivated by the parametric estimation of hidden stochastic models of the form:
\begin{equation}\label{mod1}
\left\lbrace\begin{array}{ll}
Y_{i}=X_{i}+\varepsilon_{i}\\
X_{i+1}=b_{\theta_0}(X_{i})+\sigma_{\theta_0}(X_i)\eta_{i+1},
\end{array}
\right.
\end{equation}
where one observes $Y_1$,$\cdots$,$Y_n$, and where the random variables
$\varepsilon_i$, $\eta_i$ and $X_i$ are unobserved. Notably
$(X_i)_{i \geq 0}$ is a strictly stationary, ergodic process that
depends on two measurable functions $b_{\theta_0}$
and $\sigma_{\theta_0}$ and its stationary density is $f_{\theta_0}$, where $\theta_0$ belongs to $\Theta \subset \R^p$.
The functions $b_{\theta_0}$, $\sigma_{\theta_0}$ and $f_{\theta_0}$ are known up to a finite
dimensional parameter, $\theta_0$, and the dependence with respect
to $\theta_0$ is not required to be the same in $b_{\theta_0}$ and
$\sigma_{\theta_0}$. Finally, the innovations $(\eta_i)_{i \geq
0}$ and the errors $(\varepsilon_i)_{i \geq 0}$ are independent
and identically distributed (i.i.d.) random
variables, the distribution of the innovations being known for identifiability of the model.

In this work, we propose to estimate the parameters of the two functions $b_{\theta_0}$ and
$\sigma_{\theta_0}$ driving the dynamics of  the hidden variables $(X_{i})_{i\geq 0}$. The
principle of the estimation method goes as follows. Taking that
the stationary density, $f_{\theta_0}$, is known up to the finite
dimensional parameter $\theta_0$, our M-estimator consists in
optimizing a contrast function that exploits a Fourier
deconvolution strategy in a parametric framework. In so doing, we
exploit a "Nadaraya-Watson approach" in the sense that we estimate
$b_{\theta_0}$ (respectively, $b^2_{\theta_0} + \sigma^2_{\theta_0}$) as
ratio of an estimator of $l_{\theta_0} = b_{\theta_0}f_{\theta_0}$
(respectively, $l_{\theta_0} = (b^2_{\theta_0} +
\sigma^2_{\theta_0})f_{\theta_0}$) and an estimator of $f_{\theta_0}$.
Notably we provide an analytical expression of the contrast
function for a well-known example and characterize their main properties. Moreover we show
that this deconvolution-based estimator is consistent and
asymptotically normally distributed for $\alpha$-mixing processes which leads to obtain confidence intervals in practice for many processes.
Finally, our Monte Carlo simulations show that our approach gives good results and is fast computing. All the results are illustrated on the famous stochastic volatility model with discrete time version of CIR (Cox Ingersoll Ross, see \cite{CIR85}) process for the volatility and are compared with many others methods used in the literature to estimate this model (Monte Carlo Expectation Maximisation, Sequential Monte Carlo).

Our approach extends the previous work \cite{MR3118606} where parametric estimation of models of type of \eqref{mod1} is handled for constant volatility function ($\forall x,\ \sigma_{\theta_0}(x) =\sigma_{\theta_0}$) and where the estimator proposed by the author is not adapted for stochastic process with nonlinear diffusion. As in the previous work \cite{MR3118606}, this approach is the extension to the parametric framework of the work \cite{CfLcRy10} where the authors propose a non-parametric estimation procedure in the case of discrete time stochastic models of the form of \eqref{mod1}.\footnote{See also Comte et Taupin in \cite{MR2397388}.} Our aim consists in showing how their procedure can be extended to a parametric framework and further by obtaining confidence intervals which are useful in practice.

\textbf{Applications.} This class of parametric models includes, among others, the
autoregressive model with measurement errors, the autoregressive
stochastic volatility model (\cite{Ta82}), the discrete
time versions of well-known diffusion processes in finance
(\cite{hull}, \cite{He93}) and some families of stochastic processes: Vasicek, CIR, modified CIR and hyperbolic processes (see \cite{CIR85} and \cite{GvJtLc99}). 

Here, we focus on a stochastic volatility model of the form:
\begin{eqnarray}\label{eq:appli1}
\left\lbrace\begin{array}{ll}
R_{i+1}=\exp\left(\frac{X_{i+1}}{2}\right)\xi_{i+1},\\
X_{i+1}=X_i+ \kappa (\mu-X_i)\Delta+\sigma \sqrt{\Delta X_i}\eta_{i+1}
\end{array}
\right.
\end{eqnarray} 
where $\xi_i$ and $\eta_i$ are centered gaussian random variables and $\Delta$ the sampling interval. Hence, the unobserved variance process $X_i$ is driven by a mean reverting stochastic process which was introduced in \cite{CIR85} to model the short term interest rates. 

By applying a log-transformation $Y_{i+1}=\log(R^{2}_{i+1})-\E[\log(\xi^{2}_{i+1})]$ and $\varepsilon_{i+1}=\log(\xi^{2}_{i+1})-\E[\log(\xi^{2}_{i+1})]$, the SV model  is a particular version of (\ref{mod1}) since it can be written as 
\begin{eqnarray}\label{eq:appli2}
\left\lbrace\begin{array}{ll}
Y_i=X_i+\varepsilon_{i}\\
X_{i+1}=X_i+ \kappa (\mu-X_i)\Delta+\sigma \sqrt{\Delta X_i}\eta_{i+1}
\end{array}
\right.
\end{eqnarray} 
where $\varepsilon_i$ follows a log chi-squared distribution.

From a practical point of view,  the observed component $(Y_i)_{i \geq 1}$ stands for the log-return of an asset price while the unobserved component stands for the volatility of this asset. The parameter $\kappa$ is the positive mean reverting parameter, $\mu$ is the positive long run parameter and $\sigma$ the positive volatility of the stochastic volatility process $(X_i)_{i \geq 1}$.\\

\textbf{Organization of this paper.} The paper is organized as follows. Section \ref{assum} presents the notations and the model assumptions. Section \ref{deconv} defines the deconvolution-based M-estimator and states all of the theoretical properties. Some Monte Carlo simulations are discussed in Section \ref{resultats} and some concluding remarks are provided in the last section. All the proofs can be found in Appendix \ref{A}.

\section{General setting and assumptions}\label{assum}

\noindent In this section, we introduce some preliminary main
notations and provide the assumptions of the model (\ref{mod1}).

\subsection{Notations}

Subsequently, for any function $v:\R\to \R$, we denote by $v^{*}$ the Fourier transform of the function $v$: $v^{*}(t)=\int_{}^{}e^{itx}v(x)dx$, by $||v||_{2}$ its $\mathbb{L}_2(\R)$-norm,  $||v||_{\infty}$ its supremum norm, $\langle \cdot, \cdot \rangle$ stands for the scalar product in $\mathbb{L}_2(\R)$ and ``$\star$'' for the usual convolution product. Moreover, for any integrable and square-integrable functions $u$, $u_1$, and $u_2$: we have $ (u^{*})^{*}(x)=2\pi u(-x)$and $ \left\langle u_1,u_2\right\rangle=\frac{1}{2\pi}\left\langle u^{*}_1,u^{*}_2\right\rangle$. Finally, $\left\|A\right\|$ denotes the Euclidean norm of a matrix
$A$, $\Y_{i}=(Y_{i},Y_{i+1})$ and $\y_{i}=(y_{i},y_{i+1})$,
$\mathbf{P}_n$ (respectively, $\mathbf{P}$) the empirical
(respectively, theoretical) expectation, that is, for any
stochastic variable: $\mathbf{P}_{n}(X) =
\frac{1}{n}\sum_{i=1}^{n} X_i$ (respectively,
$\mathbf{P}(X)=\E[X]$). Regarding the partial derivatives, for any
function $h_{\theta}$, $\nabla_{\theta}h_{\theta}$ is the vector of the partial
derivatives of $h_{\theta}$ with respect to (w.r.t) $\theta$ and
$\nabla^{2}_{\theta}h_{\theta}$ is the Hessian matrix of
$h_{\theta}$ w.r.t $\theta$.

\subsection{Assumptions}

We consider the hidden discrete-time diffusion model (\ref{mod1}).
The assumptions are the following.

\begin{itemize}
\item[\textbf{A0}] $\theta_0$ belongs to the interior $\Theta_0$
of a compact set $\Theta$, $\theta_0 \in \Theta \subset
\mathbb{R}^p$.

\item[\textbf{A1}] The errors $(\varepsilon_i)_{i \geq 0}$ are
independent and identically distributed centered random variables
with finit variance, $\E\left[\varepsilon_1^2\right] =
s^2_{\epsilon}$. The density of $\varepsilon_1$,
$f_{\varepsilon}$, belongs to $\mathbb{L}_2(\mathbb{R})$, and for
all $x \in \mathbb{R}$, $f^{*}_{\varepsilon}(x) \neq  0$.

\item[\textbf{A2}] The innovations $(\eta_i)_{i \geq 0}$ are
independent and identically distributed centered random variables
with unit variance $\E\left[\eta_1^2\right] = 1$ and $\E\left[\eta_1^3\right] = 0$. 

\item[\textbf{A3}] The $X_i$'s are strictly stationary and ergodic
with invariant density $f_{\theta_0}$.

\item[\textbf{A4}] The sequences $(X_i)_{i \geq 0}$ and
$(\varepsilon_i)_{i \geq 0}$ are independent. The sequence
$(\varepsilon_i)_{i \geq 0}$ and $(\eta_i)_{i \geq 0}$ are
independent.

\item[\textbf{A5}] On $\Theta_0$, the functions $\theta \mapsto
b_{\theta}$ and $\theta \mapsto \sigma_{\theta}$ admit continuous
derivatives with respect to $\theta$ up to order 2.

\item[\textbf{A6}] The function to estimate $l_{\theta} :=
\left(b_{\theta}^2 + \sigma^2_{\theta}\right)f_{\theta}$ belongs
to $\mathbb{L}_1(\mathbb{R})\cap \mathbb{L}_2(\mathbb{R})$, is
twice continuously differentiable w.r.t $\theta \in \Theta$ for
any $x$ and measurable w.r.t $x$ for all $\theta$ in $\Theta$.
Each element of $\nabla_{\theta}l_{\theta}$ and
$\nabla^{2}_{\theta}l_{\theta}$ belongs to
$\mathbb{L}_1(\mathbb{R})\cap \mathbb{L}_2(\mathbb{R})$.

\item[\textbf{A7}] The application $\theta \mapsto \mathbf{P}m_{\theta}$ admits a unique minimum and its Hessian matrix, denoted by $V_{\theta}$, is non-singular in $\theta_0$.\\
\end{itemize}

\noindent The compactness assumption \textbf{A0} might be relaxed by
assuming that $\theta_0$ is an element of the interior of a convex
parameter space $\Theta \in \mathbb{R}^{p}$. In this case, the
statistical properties of the M-estimator can be proved in the
light of convex optimization arguments. Assumptions \textbf{A1}-\textbf{A3} are
quite standard when considering estimation for convolution
models. On the other hand, Assumption $\textbf{A3}$ implies that if
$(X_i)_{i \geq 0}$ is an ergodic process then $(Y_i)_{i \geq 0}$
is stationary and ergodic since it is the sum of an ergodic
process and an i.i.d. noise process (\cite{Do94}). Consequently
$\Y_{i}=(Y_{i},Y_{i+1})$ inherits the ergodicity property.
According to Assumption \textbf{A4} the unknown density $g_{\theta_0}$ of the
$Y_i$'s is defined to be $f_{\theta_0}\star f_{\varepsilon}$. It turns
out that $g^{*}_{\theta_0} = f^{*}_{\theta_0}f^{*}_{\varepsilon}$ and
thus $f^{*}_{\theta_0} = g^{*}_{\theta_0}/f^{*}_{\varepsilon}$. Assumption $\textbf{A5}$
ensures some smoothness for the drift and diffusion functions. Assumption \textbf{A6} is also quite usual in the literature and serves for the construction and for asymptotic properties of our estimator.

\section{Parametric deconvolution estimator}\label{deconv}

\subsection{The contrast function}\label{contr}

\begin{definition}[Theoretical and empirical contrast functions] \label{contrast_function}
For any square integrable real function $v$, we set 
$$u_{v}(x)=\frac{1}{2\pi}\frac{v^{*}(-x)}{f_{\varepsilon}^{*}(x)},$$
where we recall that $f_{\varepsilon}^{*}$ is the Fourier transform of the density of the observation noise.  Let $\varphi : \R \to \R$ be given by
\begin{enumerate}
\item $\varphi:  x \mapsto x$, if $\sigma_{\theta_0}$ is a constant function of the hidden variable \label{varphicas}
\item $\varphi: x  \mapsto x^2-s_{\epsilon}^2$, if $\sigma_{\theta_0}$ is not a constant function of the hidden variable \label{varphicas2}
\end{enumerate}
where $s_{\epsilon}^2=\E[\varepsilon^2_1]$, and let us define the mapping $m$ as
\begin{eqnarray*}
m_{\cdot}(\cdot): (\theta, \z_{i})\in (\Theta \times
\mathbb{R}^{2})\mapsto
m_{\theta}(\z_i)=||l_{\theta}||^{2}-2\varphi\left(y_{i+1}\right)u^{*}_{l_{\theta}}(y_i).
\end{eqnarray*}

Then, under Assumptions \textbf{A1} up to \textbf{A7}, the
contrast function is defined by:
\begin{equation}\label{cont}
\E\left[m_{\theta}(\Y_1)\right]:=\left\|l_{\theta}\right\|^{2}-2\E\left[\varphi(Y_{2})u_{l_{\theta}}^{*}(Y_{1})\right],
\end{equation}
and its empirical couterparts is given by
\begin{equation}\label{contraste}
\mathbf{P}_{n}m_{\theta}=\frac{1}{n}\sum_{i=1}^{n}m_{\theta}(\Z_{i}).
\end{equation}
\end{definition}

\begin{remark}\label{remark:varphi}
As said in the introduction, the case where the diffusion function $\sigma_{\theta_0}$ is a constant function of the hidden variable has already been studied in \cite{MR3118606}. Therefore, from now on, we focus on the case \ref{varphicas2} in Definition \ref{contrast_function} and we refer to the aforementioned paper for the case \ref{varphicas} in Definition \ref{contrast_function}.
\end{remark}

\begin{definition}[Minimal contrast estimator]\label{definition_M_estimator}
Suppose that Assumptions \textbf{A0}-\textbf{A7} hold true then, the minimum-contrast estimator $\widehat{\theta}_n$ is defined as any solution of
\begin{equation}\label{min}
\widehat{\theta}_n=\arg\min_{\theta \in \Theta}\mathbf{P}_{n}m_{\theta}.
\end{equation}
\end{definition}

The existence of our estimator can be deduce from regularity properties of the function $l_{\theta}$ and compactness argument of the parameter space. See Appendix \ref{EoE}.\\

\begin{remark}\label{rem:smoothnessof fouriertransform}
In this paper we consider the situation in which the observation noise variance is known. This assumption, which is often not satisfied in practice, is necessary for the identifiability of the model and so is a standard assumption for state-space models given in (\ref{mod1}).\\
There is some restrictions on the distribution of the innovations in the Nadaraya-Watson approach. It is known that the rate of convergence for estimating the function $l_{\theta}$ is related to the rate of decreasing of $f^{*}_{\varepsilon}$. In particular, the smoother $f_{\varepsilon}$ is, the slower the rate of convergence for estimating is. This rate of convergence can be improved by assuming some additional regularity conditions on $l_{\theta}$ (see \cite{CfLcRy10} and \cite{MR2328553}).
\end{remark}

\begin{remark} Let us explain the choice of the contrast function and how the strategy of deconvolution works. The convergence of $\mathbf{P}_{n}m_{\theta}$ to $\mathbf{P}m_{\theta}=\E\left[m_{\theta}(\Y_{1})\right]$ as $n$ tends to the infinity is provided by the Ergodicity Theorem. Moreover, the limit $\E\left[m_{\theta}(\Y_{1})\right]$ of the contrast function can be explicitly computed. Using (\ref{mod1}) and
Assumptions \textbf{A1}-\textbf{A3}, standard computations (see Appendix \ref{appendicecontrast51}) lead to
\begin{equation*}
\E\left[m_{\theta}(\Y_1)\right]=\left\|l_{\theta}\right\|^{2}-2\left\langle
l_{\theta}, l_{\theta_{0}} \right\rangle =
\left\|l_{\theta}-l_{\theta_0}\right\|^{2}-\left\|l_{\theta_{0}}\right\|^{2},
\end{equation*}
which is, obviously, minimal at point $\theta=\theta_0$.
\end{remark}

\subsection{Asymptotic properties}

\noindent In this section we first show that our estimator is weakly consistent and asymptotically normally distributed for mixing processes. To this aim, we further assume that for $\varphi$ defined in (2) of Definition \ref{contrast_function} the following assumptions hold true: 
\begin{itemize}
\item[\textbf{A8}] (Local dominance): $\E\left[\sup_{\theta \in \Theta}\left|\varphi(Y_{i+1})u^{*}_{l_{\theta}}(Y_{i})\right|\right]<\infty$.
\item[\textbf{A9}] (Moment condition): For some $\delta >0$, $\E\left[\left|\varphi(Y_{i+1})u^*_{\nabla_{\theta}l_{\theta}}(Y_i) \right|^{2+\delta}\right]<\infty$.
\item[\textbf{A10}] (Hessian Local dominance): For some
neighbourhood $\mathcal{U}$ of $\theta_0$: $$\E\left[\sup_{\theta
\in
\mathcal{U}}\left\|\varphi(Y_{i+1})u^{*}_{\nabla_{\theta}^{2}l_{\theta}}(Y_{i})\right\|\right]<\infty$$.
\end{itemize}

\subsubsection{Asymptotic properties of the estimator: consistency and normality}

\noindent The first result regards the (weak) consistency of our estimator.

\begin{theorem} \label{consistency_theorem}
Consider the model (\ref{mod1}) under the assumptions
\textbf{A0}-\textbf{A8}, the estimator $\hat\theta_n$ defined by
(\ref{min}) is weakly consistent:
\begin{equation*}
\widehat{\theta}_{n} \longrightarrow \theta_{0}  \qquad \text{ as }n \rightarrow \infty  \text{ in }\mathbb{P}_{\theta_0}-\text{probability.} 
\end{equation*}
\end{theorem}

\begin{proof}[Sketch of proof.]
The main idea for proving the consistency of a M-estimator comes from the following observation: if $\mathbf{P}_{n}m_{\theta}$ converges to $\mathbf{P}m_{\theta}$ in probability, and if the true parameter solves the limit minimization problem, then, the limit of the argminimum $\widehat{\theta}_n$ is $\theta_0$. By using an argument of uniform convergence in probability and by compactness of the parameter space, we show that the argminimum of the limit is the limit of the argminimum. A standard method to prove the uniform convergence is to use \emph{the Uniform Law of Large Numbers} (see Lemma \ref{ULLN} in Appendix \ref{CoE}). Combining these arguments with the dominance argument \textbf{(A8)} give the consistency of our estimator, and then, the Theorem \ref{consistency_theorem}. For further details see Appendix \ref{CoE}.
\end{proof}

The second result states our estimator is $\sqrt{n}$-consistent and asymptotically normally distributed. Besides, we give in Corollary \ref{lele} the different terms of the asymptotic variance-covariance matrix. For the CLT, we need some mixing properties (we refer the reader to \cite{Do94} for a complete review of mixing processes). Hence, in the following, we further assume that:
\begin{itemize}
\item[\textbf{A11} ] The stochastic process $X_i$ is $\alpha$-mixing.
\end{itemize}

\begin{theorem}\label{CLT}
Consider the model (\ref{mod1}) under the assumptions
\textbf{A0}-\textbf{A7}, and suppose that the conditions
\textbf{A8}-\textbf{A11} hold true. Then $\hat\theta_n$ defined by
(\ref{min}) is a $\sqrt{n}$-consistent estimator of $\theta_0$
which satisfies:
\begin{equation*}
\sqrt{n}(\widehat{\theta}_{n}-\theta_{0}) \overset{\mathcal{L}}\rightarrow \mathcal{N}\left(0,\Sigma(\theta_{0})\right).
\end{equation*}
\end{theorem}

\begin{proof}[Sketch of proof.]
The asymptotic normality follows essentially from Central Limit Theorem for mixing processes (see \cite{Ga04}). Thanks to the consistency, the proof is based on a moment condition of the Jacobian vector of the function $m_{\theta}(\y)$ and on a local dominance condition of its Hessian matrix. 
For further details, see Appendix \ref{ANoE}.
\end{proof}

The following corollary gives an expression of the variance-covariance matrix $\Sigma(\theta_{0})$ of Theorem \ref{CLT} for the practical implementation:

\begin{corollary}\label{lele}
Under our assumptions, the variance-covariance matrix $\Sigma(\theta_{0})$ is given by:
\begin{eqnarray*}
\Sigma(\theta_0)=V_{\theta_0}^{-1} \Omega(\theta_0) = V_{\theta_0}^{-1}  \left[\Omega_{0}(\theta_0)+ 2\sum_{j=2}^{+\infty} \Omega_{j-1}(\theta_0)\right]V_{\theta_0}^{-1'},
\end{eqnarray*}
with

\begin{eqnarray*}
\Omega_{0}(\theta_0)&=& 4\left\{\E\left[\left(\varphi(Y_2)u^*_{\nabla_{\theta}l_{\theta}}(Y_1)
\right)\left(\varphi(Y_2)u^*_{\nabla_{\theta}l_{\theta}}(Y_1)
\right)'\right]\right.\\
&\qquad & -\left.\E\left[\left(b^2_{\theta_0}(X_1)+\sigma^2_{\theta_0}(X_1)\right) \nabla_{\theta}l_{\theta}(X_1)\right] \E\left[\left(b^2_{\theta_0}(X_1)+\sigma^2_{\theta_0}(X_1)\right) \nabla_{\theta}l_{\theta}(X_1)\right]'\right\} \\
\Omega_{j-1}(\theta_0) &=& 4\left\{\E\left[\left(b^2_{\theta_0}(X_1)+\sigma^2_{\theta_0}(X_1)\right)
\nabla_{\theta}l_{\theta}(X_1)
\left(\left(b^2_{\theta_0}(X_j)+\sigma^2_{\theta_0}(X_j)\right)\nabla_{\theta}l_{\theta}(X_j)\right)'\right]\right.\\
&\quad & \left. -\E\left[\left(b^2_{\theta_0}(X_1)+\sigma^2_{\theta_0}(X_1)\right) \nabla_{\theta}l_{\theta}(X_1)\right] \E\left[\left(b^2_{\theta_0}(X_1)+\sigma^2_{\theta_0}(X_1)\right) \nabla_{\theta}l_{\theta}(X_1)\right]'\right\}.
\end{eqnarray*} 
and the gradient $\nabla_{\theta}l_{\theta}$ is taken at point $\theta =\theta_0$; furthermore, the Hessian matrix $V_{\theta_0}$ is given by:

\begin{eqnarray*}
\left(\left[V_{\theta_0}\right]_{j,k}\right)_{1\leq j,k\leq r}&=&2\left( \left\langle \frac{\p l_{\theta}}{\p  \theta_k}, \frac{\p l_{\theta}}{\p \theta_j}\right\rangle \right)_{j,k} \text { at point $\theta=\theta_0$.} 
\end{eqnarray*}
\end{corollary}

\begin{proof}
See Appendix \ref{PC} for further details.
\end{proof}

\section{ Application: the CIR process}\label{resultats}

We consider the following stochastic volatility model
\begin{eqnarray}\label{state_space_Heston}
\left\lbrace\begin{array}{ll}
Y_i=X_i+\varepsilon_{i}\\
X_{i+1}=X_i+ \kappa (\mu-X_i)\Delta+\sigma \sqrt{\Delta X_i}\eta_{i+1}
\end{array}
\right.
\end{eqnarray} 
where $\varepsilon_i$ follows a log chi-squared distribution and $\eta_i$ a gaussian distribution. This can be also seen as a discrete version of the so called Heston model with independent noise between the log-returns $Y_i$ and its volatility $X_i$. Here, we further assume that the variance process $X_i$ is greater than zero. To ensure this condition, we make the following assumption:

\begin{equation*}
\textbf{F} \qquad a:=\frac{2\kappa\theta}{\sigma^2}\geq 1 \text{ and } c:=\frac{2\kappa}{\sigma^2}>0,
\end{equation*}
which is known as the Feller condition (see \cite{CIR85}) and implies that 
the variance process $X_i$ is ergodic and $\rho$-mixing. Furthermore, the stationary distribution $f_{\theta}$ for this process is the gamma distribution $\gamma(a,c)$ (see \cite{GvJtLc99}).

\subsection{Minimum contrast estimator}\label{modelHestDiscret}

In this case, the functions $b_{\cdot}$, $\sigma_{\cdot}$ and $l_{\cdot}$ are given by:

\begin{equation*}
b_{\theta}(x)=(1-\kappa)x+\kappa\theta, \quad \sigma_{\theta}(x)=\sigma \sqrt{x} \text{ and } l_{\theta}= \left(b^2_{\theta}(x)+\sigma^2_{\theta}(x)\right)\gamma(a,c)
\end{equation*}
$\forall x \in \R^{*}_{+}$ with $\theta=(\kappa, \mu, \sigma)$. Using the Fourier transform of the Gamma and the log chi-squared density, we have

\begin{equation*}
f^{*}_{\theta}(x)=\left(1-\frac{ix}{c}\right)^{-a} \text{ and }
f^{*}_{\varepsilon}(x)=\frac{1}{\sqrt{\pi}}2^{ix}\Gamma(\frac{1}{2}+ix)\exp(-iCx),
\end{equation*}
with $C$ the expectation of the logarithm of a chi-squared random variable, \emph{i.e.} $C=-1.27$ (see \cite{abram} and Appendix \ref{appendice2} for the expression of the Fourier transform). Next, the Fourier transform of $l_{\theta}$ is given by
\begin{equation*}
l_{\theta}^{*}(x)=-\alpha_1\left[\frac{-a}{c^{2}}(a+1)\left(1-\frac{ix}{c}\right)^{-a-2}\right]+i\alpha_2\frac{a}{c}\left(1-\frac{ix}{c}\right)^{-a-1}+\alpha_3\left(1-\frac{ix}{c}\right)^{-a}.
\end{equation*}
with $\alpha_1=\left(1-\kappa\right)^2,\ \alpha_2=2\left(1-\kappa\right)\kappa\theta+\sigma^2$, $\ \alpha_3=(\kappa\theta)^2$. Finally, the $\mathbb{L}_2$-norm of $l_{\theta}$ is given by:

\begin{eqnarray*}
\left\|l_{\theta}\right\|^2_{2} &=& \alpha_1^2 2^{-(2a+3)}c^{-3}\frac{\Gamma(2a+3)}{\Gamma^{2}(a)} + 2\alpha_1\alpha_2 2^{-(2a+2)}c^{-2}\frac{\Gamma(2a+2)}{\Gamma^{2}(a)}\\
&&\quad + \left(2\alpha_1\alpha_3+\alpha_2^2\right) 2^{-(2a+1)}c^{-1}\frac{\Gamma(2a+1)}{\Gamma^{2}(a)} + \alpha_2\alpha_3 2^{-(2a)}\frac{\Gamma(2a)}{\Gamma^{2}(a)}\\
&& \quad + \alpha_3^2 2^{-2a+1}c\frac{\Gamma(2a-1)}{\Gamma^{2}(a)}.
\end{eqnarray*}
where $\Gamma$ corresponds to the Gamma function given by $\Gamma(z)= \int_{\R_{+}}t^{z-1}\exp(-t)dt$.
\begin{proof} See Appendix \ref{appendice2}.
\end{proof}


Hence, the M-estimator solves:

\begin{equation}\label{contraste_application3}
\widehat{\theta}_n= \arg \min_{\theta \in \Theta}\left\{ \left\|l_{\theta}\right\|^2_{2}-\frac{2}{n}\sum_{i=1}^{n}Y_{i+1}u^{*}_{l_{\theta}}(Y_{i})\right\}
\end{equation}
where: 
\begin{equation*}
u_{l_{\theta}}(y)=\sqrt{\pi}\left(\frac{\alpha_1\left[\frac{a}{c^{2}}(a+1)\left(1-\frac{iy}{c}\right)^{-a-2}\right]+i\alpha_2\frac{a}{c}\left(1-\frac{iy}{c}\right)^{-a-1}+\alpha_3\left(1-\frac{iy}{c}\right)^{-a}}{2^{i\beta y}\Gamma(\frac{1}{2}+i\beta y )\exp(-i\tilde{C}y)}\right)
\end{equation*}

\subsection{Others methods}\label{comp-time}

\textbf{Particle filters :  EKF, APFS, APF and KSAPF.} For the comparison with our contrast estimator given in (\ref{contraste_application3}), we use the following methods: the Extended Kalman Filter (EKF), the Auxiliary Particle filter (APF), the Auxiliary Particle filter with static parameter (APFS) and the Kernel Smoothing Auxiliary Particle filter (KSAPF). We refer the reader to \cite{MR1684426}, \cite{PiSh07}, \cite{DoFr01} and \cite{liu} for a complete revue of these methods.\\

In order to estimate the parameters with these methods, for the EKF, APF and KSAPF estimators we use the Kitagawa and al.'s approach (see \cite{DoFr01} chapter 10 p.189) in which the parameters are supposed time-varying: $\theta_{i+1}=\theta_{i}+\mathcal{G}_{i+1}$ where $\mathcal{G}_{i+1}$ is a centered Gaussian random with a variance matrix $Q$ supposed to be known. Hence, we consider the augmented state vector $\tilde{X}_{i+1}=(X_{i+1}, \theta_{i+1})'$ where $X_{i+1}$ is the hidden state variable and $\theta_{i+1}$ the unknown vector of parameters. Furthermore, for the numerical part we call APFS an Auxiliary Particle filter without time-varying parameters. \\

\textbf{The MCEM.} From a theoretical point of view, the MLE is asymptotically efficient. However, in practice since the states $(X_{1}\cdots, X_{n})$ are unobservable and since the model \eqref{state_space_Heston} is non Gaussian, the likelihood is untractable. We have to use numerical methods to approximate it. In this section, we illustrate the MCEM estimator which consists in approximating the likelihood and applying the Expectation-Maximisation algorithm introduced by Dempster \cite{dempster} to find the parameter $\theta$.\\

\subsection{Numerical Results}

 In this section we present some Monte Carlo simulations using the model \eqref{state_space_Heston}. For the analysis we consider the following ``true parameter'' $\theta_0=(\kappa_0, \mu_0, \sigma_0^{2})=(4,0.03,0.4)$ which is consistent with empirical applications of daily data (see \cite{bin}). Thus, we have sampled the trajectory of the $X_i$, and conditionally to the trajectory, we have sampled the variables $Y_i$ with a variance noise $s^{2}_{\varepsilon}=0.1$ \footnote{ For the simulation (see Appendix \ref{appendice2}) we take $f^{*}_{\varepsilon}(x)=\frac{1}{\sqrt{\pi}}2^{i\beta x}\Gamma(\frac{1}{2}+i\beta x)\exp(-i\tilde{C}x)$ with $\beta=\sqrt{2s_{\varepsilon}^2/\pi^2}$ and $\tilde{C}=\beta C$ where $C=-1.27$.}

The numerical illustration goes as follows: we work with a number of observations $n$ equal to $1000$, we first compare all the methods proposed in term of computing time, then we run $N=100$ estimates for each method and we compare the performance of our estimator with others methods by computing the Mean Square Error (MSE) defined as:

\begin{equation}\label{MSE}
{\rm MSE}=\frac{1}{N}\sum_{i=1}^{p}\sum_{j=1}^{N}(\hat{\theta}_{j}^{i}-\theta_{0}^{i})^2,
\end{equation}
where $p$ corresponds to the dimension of the vector of parameters. 

Then, we illustrate the statistical properties of our contrast estimator, by computing the confidence intervals for different number of observations. Finally, we study the influence of the signal noise ratio since the performance of our estimator depends on the regularity of $l^{*}_{\theta}$ and $f^{*}_{\varepsilon}$ (see Remark \ref{rem:smoothnessof fouriertransform}).
\\

For particles methods, we take a number of particles $M$ equal to $5000$.  Note that for the Bayesian procedure (APF, APFS and KSAPF) we need a prior on $\theta$, and this only at the first step. The prior for $\theta$ is taken to be the Uniform law and conditionally to $\theta$ the distribution of $X_1$ is its stationary law:

\begin{equation*}
\left\lbrace\begin{array}{ll}
p(\theta)=\mathcal{U}(3, 5)\times \mathcal{U}(0.02, 0.04) \times \mathcal{U}(0.3, 0.5) \\
f_{\theta}(X_1)=\gamma\left(\theta \right)
\end{array}
\right.
\end{equation*}
For the KSAPF, we take a bandwidth $h=0.1$ and for the APF we take a matrix $Q=[10^{-3}, 0.1.10^{-4}, 10^{-4}]\mathbf{I}_{3}$ with $\mathbf{I}_{3}$ the identity matrix in $\R^{3}$ (see Section \ref{comp-time} for the definition of the matrix $Q$). 

\subsubsection{Computing time}\label{comp-time}

\begin{table}[h!]
\caption{\label{time_gaussian} Comparison of the CPU for Particle filters estimators, MCEM estimator and Contrast estimator. }
\begin{center}
\begin{tabular}{|c|c|c|c|c|c|c|}
\hline
 & \multicolumn{1}{|c|} {APF}  & \multicolumn{1}{|c|} {KSAPF} & \multicolumn{1}{|c|} {APFS}&  \multicolumn{1}{|c|} {EKF}  & Contrast  &  \multicolumn{1}{|c|} {MCEM} \\
\hline
CPU (sec) & 105.1695  & 93.8846  & 192.2166  & 0.2 &  20.4074 & 217430\\
 \hline
\end{tabular}
\end{center}
\end{table}
This comparison illustrates the numerical complexity of the MCEM. Therefore, in the following, we only compare our contrast estimator with the particles estimators.\\

\subsubsection{Parameter estimates} 
We  illustrate by boxplots the different estimates (see Figures [\ref{compar_param_partB3A}] up to [\ref{compar_param_partB3C}]). In Table [\ref{resu_Heston_APF_EKFAB}], we compute the MSE defined in \eqref{MSE} for each method and with a number of MC equal to $N=100$ and the CPU for a number of observations $n=1000$ (see Table [\ref{time_gaussian}]). \\

We note that for all parameters, the EKF estimator is very bad since the stochastic volatility model \eqref{state_space_Heston} is strongly nonlinear, and its corresponding boxplots have the largest dispersion meaning that this filter is not stable and not appropriated to estimate this model. Among particle filters, the KSAPF and the APF are the best estimators although the dispersion is huge for the mean reversion parameter $\kappa$ and the volatility parameter $\sigma$.\\ Besides, the APFS is less efficient than the others particle filters. Our estimator is stable and performs the others when one compare the MSE. From a computational point of view, all particles filters have an equivalent CPU. Our contrast estimator is fast and its implementation is straightforward and the MSE is the smallest (see Table [\ref{time_gaussian}]). 

`
\begin{table}[h!]
\caption{\label{resu_Heston_APF_EKFAB} Comparison of the $MSE(\theta)$ for Particle filters and Contrast estimator. }

\begin{center}
\begin{tabular}{|c|c|c|c|c|c|}
\hline
 & \multicolumn{1}{|c|} {APF}  & \multicolumn{1}{|c|} {KSAPF} & \multicolumn{1}{|c|} {APFS}&  \multicolumn{1}{|c|} {EKF}  & Contrast  \\
\hline
MSE$(\theta)$ & 0.189  & 0.166  & 0.205  & 0.43 &  0.124 \\
 \hline
\end{tabular}
\end{center}
\end{table}
\begin{figure}[h!]
\begin{center}
\includegraphics[width=139mm, height=70mm]{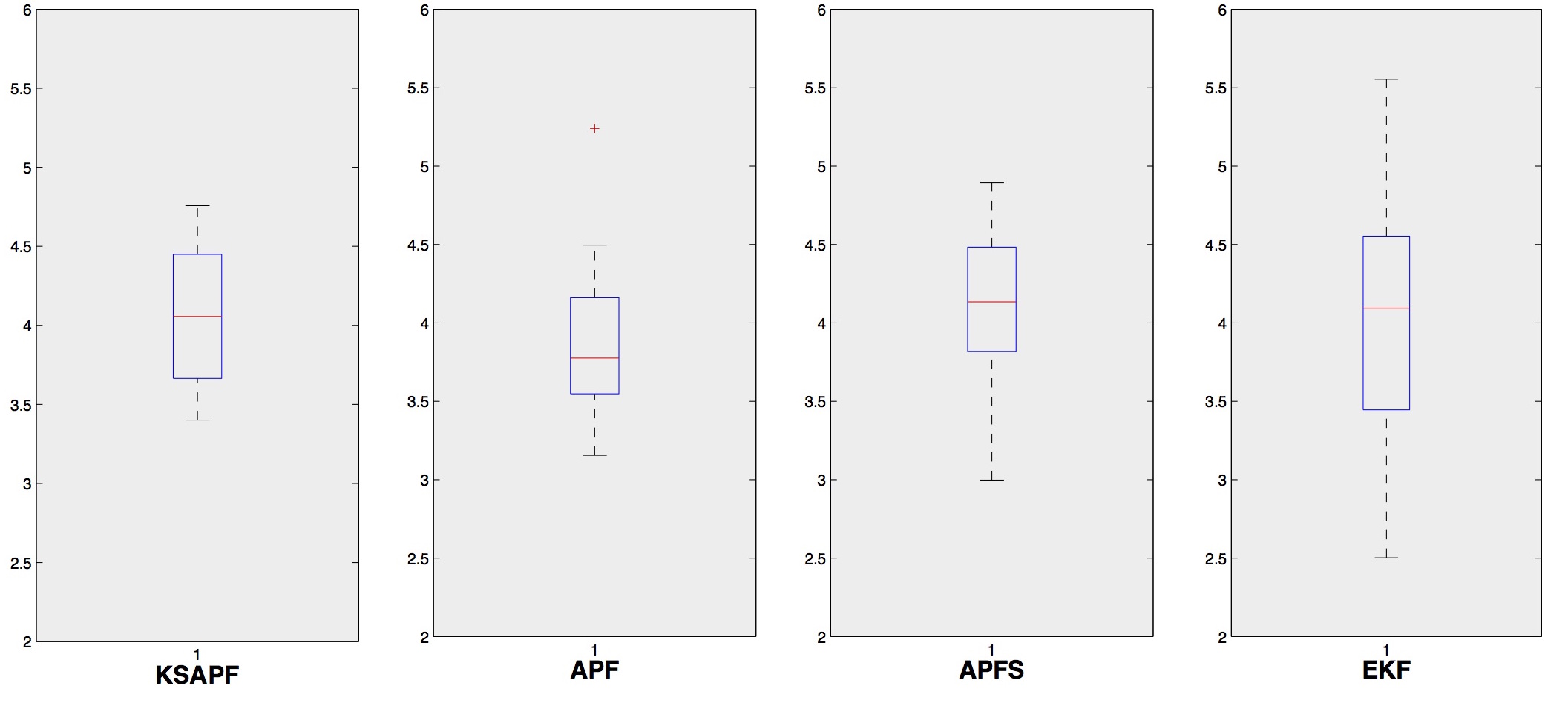}\includegraphics[width=32.25mm, height=70mm]{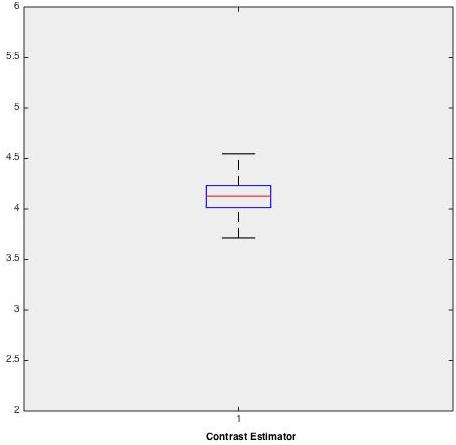}
\end{center}
\caption{ \label{compar_param_partB3A} Boxplot of the parameter $\kappa$. True value equal to $4$.}
\end{figure}

\begin{figure}[h!]
\begin{center}
\includegraphics[width=139mm, height=70mm]{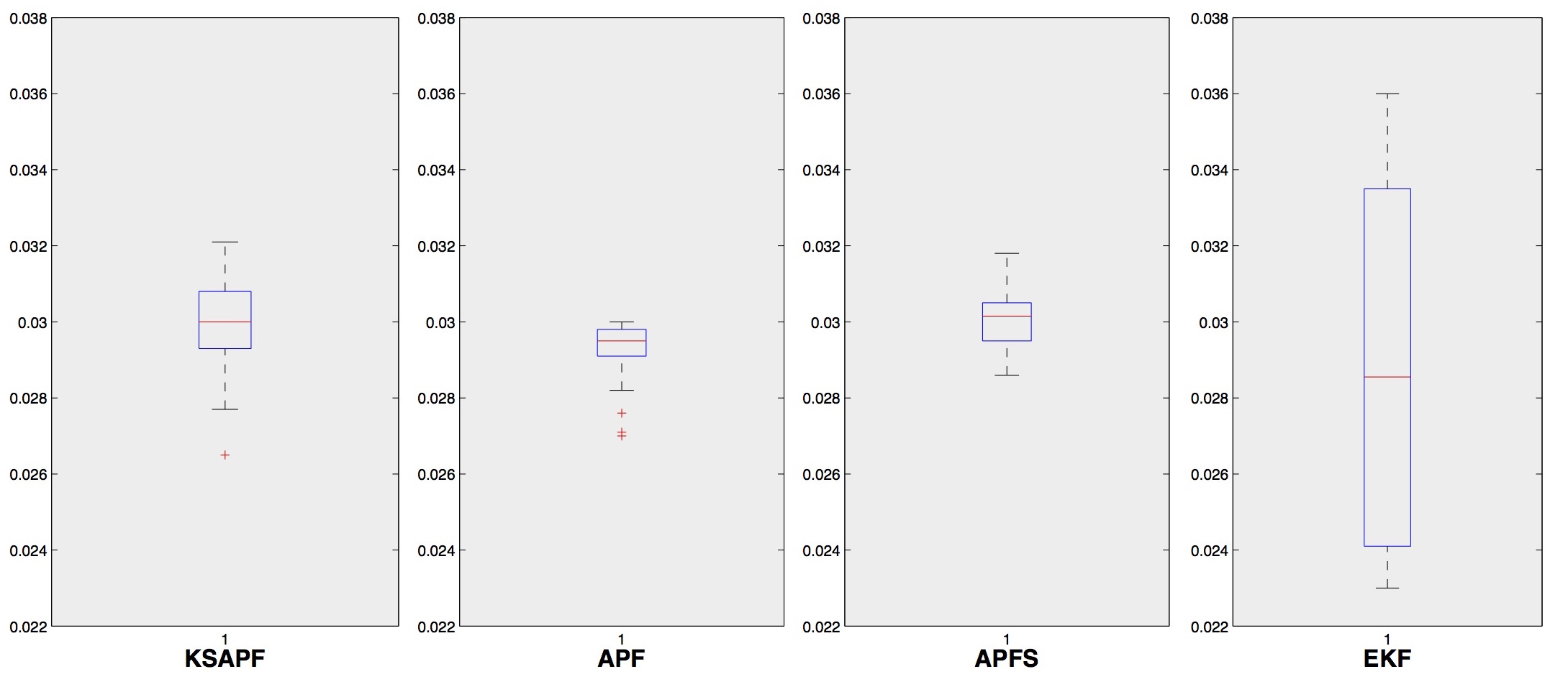}\includegraphics[width=32.25mm, height=70mm]{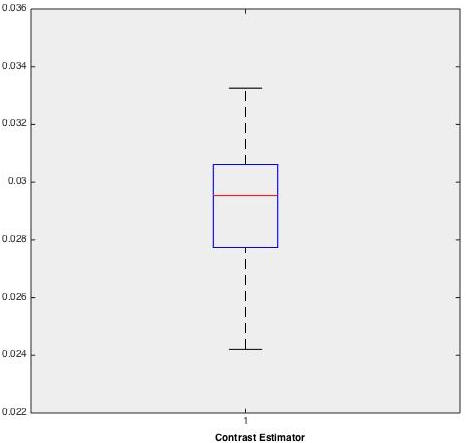}
\end{center}
\caption{ \label{compar_param_partB3B} Boxplot of the parameter $\mu$. True value equal to $0.03$.}
\end{figure}

\begin{figure}[h!]
\begin{center}
\includegraphics[width=139mm, height=70mm]{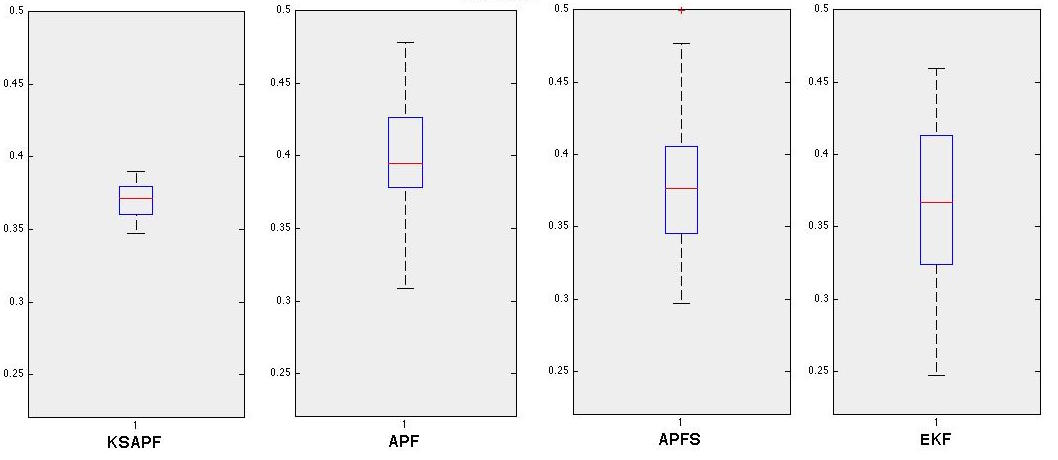}\includegraphics[width=32.25mm, height=70mm]{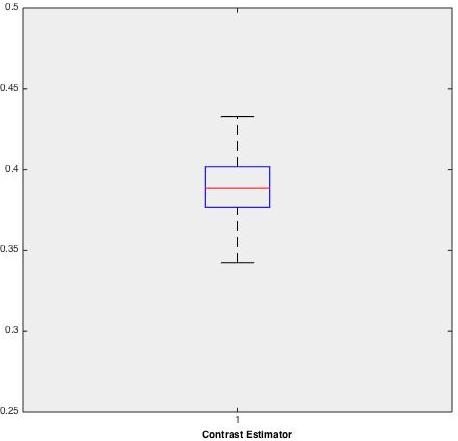}
\end{center}
\caption{ \label{compar_param_partB3C} Boxplot of the parameter $\sigma^2$. True value equal to $0.4$.}
\end{figure}
\newpage

\subsubsection{Confidence Interval of the minimum contrast estimator}

To illustrate the statistical properties of our contrast estimator, we  compute the confidence intervals with the confidence level $1-\alpha$ equal to $0.95$ for $N=100$ estimators. The coverage corresponds to the number of times for which the true parameter $\theta_{0}^i, i=1,\cdots, p$ belongs to the confidence interval. The results are illustrated in Figure [\ref{Covera1}]. We note that the coverage converges to $95\%$ for a small number of observations and as expected, the confidence interval decreases with the number of observations. Note that of course a MLE confidence interval would be smaller since the MLE is efficient but the corresponding computing time would be huge (see Table [\ref{time_gaussian}]). 

\begin{figure}[h!]
\begin{center}
\includegraphics[width=149mm, height=50mm]{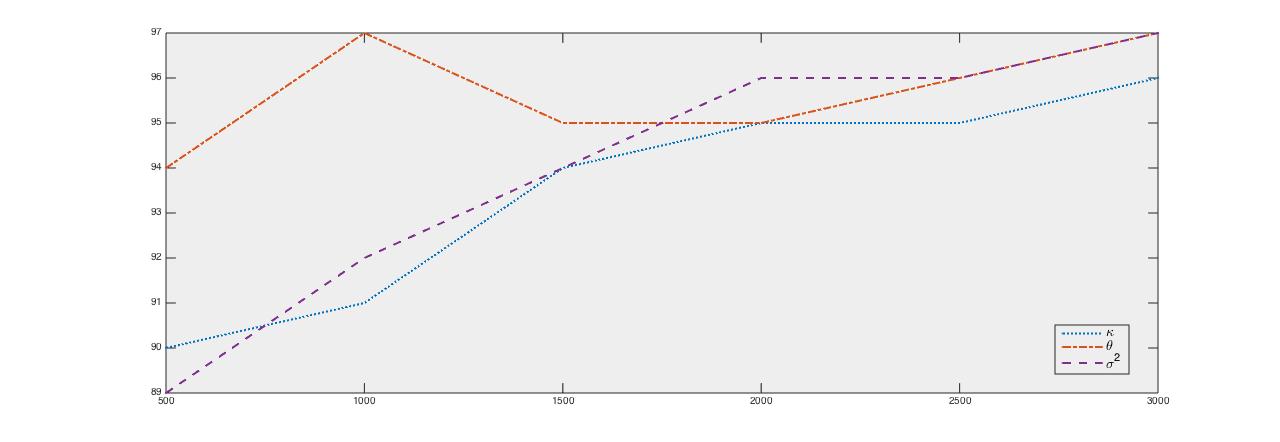}
\end{center}
\caption{\normalsize Coverage with respect to the number of observations $n=500$ up to $3000$ for $N=100$ estimators.}\label{Covera1} 
\end{figure}

  \newpage    

\subsubsection{Ratio signal-noise for the contrast estimator}

We denote by $r=[s^{2}_{\varepsilon}/\sigma^{2}]$ the ratio signal-noise and in Table (\ref{compar3}) we compare the MSE for different $r$ and different number of observations $n$ for the contrast estimator. We note that the MSE decreases with the number of observations and is smaller for small ratio-signal-noise. As explained in Section \ref{contr} and see \cite{CfLcRy10} for more details, the rate of convergence of our approach depends on the regularity of the noise density $f_{\varepsilon}$. And, in particular, the smoother the noises are, the slower the rate of convergence is. For the CIR model, the density of the noises and the function $l_{\theta}$ are ordinary smooth, so we are in a favourable case.

\begin{table}[h!]
\begin{center}
\captionof{table}{Ratio signal-noise for the estimation of the CIR model}\label{compar3}
\begin{tabular}{|c|c|c|c|c|}
\hline
\quad &  \multicolumn{1}{c|}{Mean($\hat{\mu}_n)$}  &  \multicolumn{1}{c|}{Mean($\hat{\kappa}_n)$} & \multicolumn{1}{c|}{Mean($\hat{\sigma}^2_n)$} &  MSE \\
\hline
$n=500$ and $r=0.1$ & 0.0315 & 3.88 & 0.401 & 0.14\\
\hline
$n=500$ and $r=1$ & 0.0303 & 3.89 & 0.405 & 0.16 \\
\hline
$n=1000$ and $r=0.1$  & 0.0312 &  3.76 & 0.401 & 0.11 \\
\hline
$n=1000$ and $r=1$ & 0.0308 &  3.83 & 0.41 & 0.18\\
\hline
\end{tabular}
\end{center}
\end{table}

$\newline$

\subsection{Summary and Conclusions}

In this paper we have proposed a new method to estimate hidden nonlinear diffusion process. This method is based on a deconvolution strategy and leads to consistent and asymptotically normal estimator. We have numerically studied the performance of our estimator for the CIR process widely used in many domains and we were able to construct confidence interval (see Figure [\ref{Covera1}]). As the boxplots [\ref{compar_param_partB3A}] up to [\ref{compar_param_partB3C}] show, only Contrast, APF, and KSAPF estimators are comparable. Indeed EKF and APFS estimators are biased and their MSE are bad, especially for the EKF method since the CIR process is nonlinear.  Furthermore, if one compares the MSE of the particle filters, the KSAPF estimator is the best method. Among particles filters, it is clearly known that the APFS is less efficient than the APF filter since the parameters are not time-varying and so the only randomness is made at the first step by the prior law and not in each propagation step.\\
 Then, the Contrast, APF, and KSAPF methods lead to unbiased and not so much varying estimator. We emphasize that our estimator performs the others in a MSE aspect (see Table [\ref{time_gaussian}]). Most importantly, our estimator can be constructed without any arbitrary parameters choice, is straightforward to implement, fast and allows to construct confidence interval.\\

\bibliographystyle{apalike}
\bibliography{bibio_Mestimat_Ext}

\section{Appendix}\label{A}

\subsection{The contrast function}\label{appendicecontrast51}

\textbf{The procedure:} The limit $\E\left[m_{\theta}(\Y_{1})\right]$ of the contrast function can be explicitly computed. Using (\ref{mod1}) and
Assumptions \textbf{A1}-\textbf{A3}, we obtain:
\begin{eqnarray*}
\E\left[ (Y_2^2-s_{\varepsilon}^2)  u^*_{l_{\theta}}(Y_1)\right] &=& \E\left[ \left(X_2^2+ 2X_2\varepsilon_2 + \varepsilon_2^2 - s_{\varepsilon}^2 \right)  u^*_{l_{\theta}}(Y_1)\right]\\
&=&\E\left[ X_2^2  u^*_{l_{\theta}}(Y_1)\right] \text{ by assumption \textbf{A1}}\\
&=& \E\left[ \left(b^2_{\theta_{0}}(X_1)+ \sigma^2_{\theta_{0}}(X_1)\eta_2^2 + 2b_{\theta_{0}}(X_1)\sigma_{\theta_{0}}(X_1)\eta_2 \right)  u^*_{l_{\theta}}(Y_1)\right]\\
&=&  \E\left[ \left(b^2_{\theta_{0}}(X_1)+
\sigma^2_{\theta_{0}}(X_1)\right)  u^*_{l_{\theta}}(Y_1)\right]  \text{ by assumption \textbf{A2}},
\end{eqnarray*}
Using Fubini's Theorem and (\ref{mod1}), it follows that:
\begin{eqnarray*}
\E\left[ \left(b^2_{\theta_{0}}(X_1)+
\sigma^2_{\theta_{0}}(X_1)\right)  u^*_{l_{\theta}}(Y_1)\right] &=& \E\left[\left(b^2_{\theta_{0}}(X_1)+
\sigma^2_{\theta_{0}}(X_1)\right) \int e^{iY_1y} u_{l_{\theta}}(z) dz \right]\nonumber\\
&=&\E\left[\left(b^2_{\theta_{0}}(X_1)+
\sigma^2_{\theta_{0}}(X_1)\right) \int \frac{1}{2\pi}\frac{1}{f_{\varepsilon}^*(z)}e^{iY_1z} (l_{\theta}(-z))^*dy  \right]\nonumber\\
&=&\frac{1}{2\pi} \int \E\left[\left(b^2_{\theta_{0}}(X_1)+
\sigma^2_{\theta_{0}}(X_1)\right)e^{i(X_1+\varepsilon_1)z} \right] \frac{1}{f_{\varepsilon}^*(z)} (l_{\theta}(-z))^* dz \nonumber\\
&=&\frac{1}{2\pi} \int \frac{\E\left[ e^{i\varepsilon_1 z}\right]}{f_{\varepsilon}^*(z)} \E\left[\left(b^2_{\theta_{0}}(X_1)+
\sigma^2_{\theta_{0}}(X_1)\right)e^{iX_1 z}\right] (l_{\theta}(-z))^* dy\nonumber\\
&=&\frac{1}{2\pi} \E\left[\left(b^2_{\theta_{0}}(X_1)+
\sigma^2_{\theta_{0}}(X_1)\right) \int e^{iX_1 z} (l_{\theta}(-z))^* dz \right]\nonumber\\
&=&\frac{1}{2\pi} \E\left[ \left(b^2_{\theta_{0}}(X_1)+
\sigma^2_{\theta_{0}}(X_1)\right) \left((l_{\theta}(-X_1))^{*}\right)^*\right]\nonumber\\
&=& \E\left[ \left(b^2_{\theta_{0}}(X_1)+
\sigma^2_{\theta_{0}}(X_1)\right)l_{\theta}(X_1)\right].\nonumber\\
&=&\int \left(b^2_{\theta_{0}}(x)+
\sigma^2_{\theta_{0}}(x)\right)f_{\theta_{0}}(x)\left(b^2_{\theta}(x)+
\sigma^2_{\theta}(x)\right)f_{\theta}(x)dx \nonumber\\
&=&\left\langle l_{\theta}, l_{\theta_{0}} \right\rangle.
\end{eqnarray*}
Then,
\begin{equation*}
\E\left[m_{\theta}(\Y_1)\right]=\left\|l_{\theta}\right\|^{2}-2\left\langle
l_{\theta}, l_{\theta_{0}} \right\rangle =
\left\|l_{\theta}-l_{\theta_0}\right\|^{2}-\left\|l_{\theta_{0}}\right\|^{2}.
\end{equation*}

\subsection{Proofs}\label{appendice3}

For the reader convenience we split the proof of Theorems \ref{consistency_theorem} and \ref{CLT} into three parts: in Subsection \ref{EoE}, we give the proof of the existence of our contrast estimator defined in (\ref{min}). In Subsection \ref{CoE}, we prove the consistency, that is, the Theorem \ref{consistency_theorem}. Then, we prove the asymptotic normality of our estimator in  Subsection \ref{ANoE}, that is, the Theorem \ref{CLT}.  The Subsection \ref{PC} is devoted to  Corollary \ref{lele}. \\

Recall from Remark \ref{remark:varphi} that we only made the proof for the function $\varphi$ defined by \eqref{varphicas2} in Definition \ref{contrast_function} and we refer to \cite{MR3118606} for the proof in the case \eqref{varphicas} of Definition \ref{contrast_function}. 

\subsubsection{ Proof of the existence and measurability of the M-Estimator}\label{EoE}

By assumption, the function $\theta \mapsto \left\|l_{\theta}\right\|_{2}^2$ is continuous. Moreover, $l^{*}_{\theta}$ and then $u^{*}_{l_{\theta}}(x)=\frac{1}{2\pi}\int e^{ixy}\frac{l^{*}_{\theta}(-y)}{f^{*}_{\varepsilon}(y)}dy$ are continuous w.r.t $\theta$. In particular, the function $m_{\theta}(\y_{i})=\left\|l_{\theta}\right\|_{2}^2-2\varphi(y_{i+1})u^{*}_{l_{\theta}}(y_{i})$ is continuous w.r.t $\theta$, for $\varphi: x \mapsto x^2-s_{\varepsilon}^2$. Hence, the function $\mathbf{P}_{n}m_{\theta}=\frac{1}{n}\sum_{i=1}^{n}m_{\theta}(\Y_i)$ is continuous w.r.t. $\theta$ belonging to the compact subset $\Theta$. So, there exists $\tilde{\theta}$ that belongs to $\Theta$ such that: 
 \begin{equation*}
\inf_{\theta \in \Theta}\mathbf{P}_{n}m_{\theta}=\mathbf{P}_{n}m_{\tilde{\theta}}.\qed
 \end{equation*}
 
\subsubsection{ Proof of the Consistency}\label{CoE}

For the consistency of our estimator, we need to use the uniform convergence given in the following Lemma. Let us consider the following quantities:

\begin{equation*}
\mathbf{P}_{n}h_{\theta}=\frac{1}{n}\sum_{i=1}^{n}h_{\theta}(Y_i);\quad \mathbf{P}_{n}S_{\theta}=\frac{1}{n}\sum_{i=1}^{n}\nabla_{\theta}h_{\theta}(Y_i) \text{ and }  \mathbf{P}_{n}H_{\theta}=\frac{1}{n}\sum_{i=1}^{n}\nabla^{2}_{\theta}h_{\theta}(Y_i)
\end{equation*}

where $h_{\theta}$ is real function from $\Theta \times \mathcal{Y}$ with value in $\R$.\\

\begin{lemma}{Uniform Law of Large Numbers (\textbf{ULLN})(see \cite{newey} for the proof.)}\label{ULLN}\qquad\\

Let $(Y_i)$ be an ergodic stationary process and suppose that:
\begin{enumerate}
\item  For all $y \in \R$, $\theta \mapsto h_{\theta}(y)$ is continuous and for all $\theta \in \Theta$, $y \mapsto h_{\theta}(y)$ is measurable.
\item   There exists a function (called the dominating function) $s:\R\to \R$ such that for all $\theta \in \Theta$ $\left|h_{\theta}(y)\right|\leq s(y)$ and $\E[s(Y_1)]<\infty$. Then:
\end{enumerate}

\begin{equation*}
\sup_{\theta \in \Theta}\left|\mathbf{P}_{n}h_{\theta}-\mathbf{P}h_{\theta}\right|\rightarrow 0 \qquad \text{ in probability as n } \rightarrow  \infty. 
\end{equation*}

Moreover, $\mathbf{P}h_{\theta}$ is a continuous function of $\theta$.\\
\end{lemma} 

By assumption for all $x$, $\theta \mapsto l_\theta(x)$ is continuous and for all $\theta$, $x \mapsto l_{\theta}(x)$ is measurable which implies the continuity and the measurability of the function $\mathbf{P}_{n}m_{\theta}$ on the compact subset $\Theta$. Furthermore, the local dominance assumption \textbf{(A8)} implies that $\E\left[\sup_{\theta \in \Theta }\left|m_{\theta}(\Y_i)\right|\right]$ is finite. Indeed,

\begin{eqnarray*}
\left|m_{\theta}(\y_{i})\right|&=&\left|\left\|l_{\theta}\right\|_{2}^2-2\varphi(y_{i+1})u^{*}_{l_{\theta}}(y_{i})\right|\nonumber\\
&\leq& \left\|l_{\theta}\right\|_{2}^2+2\left|\varphi(y_{i+1})u^{*}_{l_{\theta}}(y_{i})\right|.\label{dom}
\end{eqnarray*}
with $\varphi$ the function defined in (2) in Definition \ref{contrast_function}.

As $\left\|l_{\theta}\right\|_{2}^2$ is continuous on the compact subset $\Theta$, $\sup_{\theta \in \Theta }\left\|l_{\theta}\right\|_{2}^2$ is finite. Therefore, $\E\left[\sup_{\theta \in \Theta }\left|m_{\theta}(\Y_i)\right|\right]$ is finite if $\E\left[\sup_{\theta \in \Theta }\left|\varphi(Y_{i+1})u^{*}_{l_{\theta}}(Y_i)\right|\right]$ is finite. Lemma \ref{ULLN} gives us the uniform convergence in probability of the contrast function: for any $\varepsilon>0$,

\begin{equation*}
\lim_{n\rightarrow +\infty} \mathbb{P}\left(\sup_{\theta \in \Theta}\left|\mathbf{P}_{n}m_{\theta}-\mathbf{P}m_{\theta}\right|\leq \varepsilon\right)=1.
\end{equation*}

Combining the uniform convergence with Theorem 2.1 p. 2121 chapter 36 in \cite{robert} yields the weak (convergence in probability) consistency of the estimator.\qed

\subsubsection{Proof of the asymptotic normality}\label{ANoE}

Consider the model (\ref{mod1}) under the assumptions
\textbf{A0}-\textbf{A7}. The proof of the asymptotic normality
results from assumptions \textbf{A8}-\textbf{A11} is a
straighforward application of \cite{Fy00} (Propostion 7.8. p.
472 and \cite{Ga04}). Furthermore, for the CLT, we need to recall some mixing properties (we refer the reader to \cite{Do94} for a complete revue of mixing processes). \\

Let $K_{\theta}(x, dy)$ be a Markov transition kernel on a general space $(\mathcal{X}, \mathcal{B}(\mathcal{X})))$ and let $K^{n}_{\theta}(x, dy)$ denotes the $n$ step Markov transition corresponding to $K_{\theta}$. Then, for $k \in \mathbb{N}, x \in \mathcal{X}$ and a measurable set $\mathcal{A}$:

\begin{equation*}
K^{n}_{\theta}(x, dy)=\mathbb{P}_{\theta}(X_{n+k} \in \mathcal{A} \vert X_k=x)
\end{equation*}

Let $M:\R \to \R$ be a nonnegative function and $\gamma : \mathbb{N} \to \mathbb{Z}_{+}$ be a nonnegative decreasing function such that:

\begin{equation*}
\mathbf{(C)} \qquad \vert\vert K^{n}_{\theta}(x, .)-f_{\theta}(.)\vert \vert_{VT} \leq M(x) \gamma(n)
\end{equation*}
where $\vert\vert.\vert\vert_{VT}$ denotes the total variation norm.\\

\begin{remark}
$X$ is geometrically ergodic if $\mathbf{(C)}$ holds with $\gamma(n)=t^{n}$ for some $t <1$.
$X$ is uniform ergodic  if $\mathbf{(C)}$ holds with $M$ bounded and $\gamma(n)=t^{n}$ for some $t <1$.  
$X$ is polynomial ergodic of order $m$ where $m\geq 0$ if $\mathbf{(C)}$ holds with $\gamma(n)=n^{-m}$.
\end{remark}

The proof of the asymptotic normality is based on the following Lemma: 

\begin{lemma}\label{F}[see \cite{Fy00} and \cite{Ga04} for the proof.]

Suppose that the conditions of the consistency hold. Suppose further that:

\begin{enumerate}
\item $\Y_i$ is $\alpha$-mixing.
\item (Moment condition): for some $\delta >0$ and for each $j\in \left\{1,\cdots ,p\right\}:$
$$\E\left[\left|\frac{\p m_{\theta}(\Y_{1})}{\p \theta_j}\right|^{2+\delta}\right]<\infty$$.
\item Assumption $\mathbf{(C)}$ holds such that $\E[M(X_1)]<\infty$ and $\gamma(n) $ satisfies $\sum_{n}^{}\gamma(n)^{\frac{\delta}{2+\delta}}<\infty$ with $\delta$ defined in (2).
\item (Hessian Local condition): for some neighbourhood $\mathcal{U}$ of $\theta_0$ and for $j,k\in \left\{1,\cdots, p\right\}$ 
$$\E\left[\sup_{\theta \in \mathcal{U} }\left|\frac{\p^2m_{\theta}(\Y_{1})}{\p \theta_j \p \theta_k}\right|\right]< \infty.$$
\end{enumerate}

Then, $\widehat{\theta}_{n}$ defined in Eq.(\ref{min}) is asymptotically normal with asymptotic covariance matrix given by:
\begin{equation*}
\Sigma(\theta_{0})=V_{\theta_0}^{-1} \Omega(\theta_{0})V_{\theta_0}^{-1}
\end{equation*}
where $V_{\theta_0}$ is the Hessian of the application $\mathbf{P}m_{\theta}$ given in Eq.(\ref{cont}).
\end{lemma}

\begin{proof}

It just remains to check that the conditions (2) and (4) of Lemma \ref{F} hold under our assumptions.\\

\noindent \paragraph{\small{Moment condition}:}
as the function $l_{\theta}$ is twice continuously differentiable w.r.t $\theta$, for all $\y_{i}$ $\in$ $\R^2$, the application $m_{\theta}(\y_{i}):\ \theta \in \Theta \mapsto m_{\theta}(\y_{i})=||l_{\theta}||_{2}^2 - 2\varphi(y_{i+1})u^*_{l_{\theta}}(y_{i})$ is twice continuously differentiable for all $\theta$ $\in$ $\Theta$ and its first derivatives are given by:

\begin{equation*}
\nabla_{\theta}m_{\theta}(\y_{i})= \nabla_{\theta}||l_{\theta}||_{2}^2 - 2\varphi(y_{i+1})\nabla_{\theta}u^*_{l_{\theta}}(y_{i}).
\end{equation*}

 By assumption, for each $j\in \left\{1,\cdots,p\right\}$, $\frac{\p l_{\theta}}{\p \theta_j} \in \mathbb{L}_{1}(\R)$, therefore one can apply the Lebesgue differentiation Theorem and Fubini's Theorem to obtain :

\begin{equation}
\nabla_{\theta}m_{\theta}(\y_{i})
=\left[\nabla_{\theta}||l_{\theta}||_{2}^2 - 2\varphi(y_{i+1})u^*_{\nabla_{\theta}l_{\theta}}(y_{i}) \right]\label{expres}.
\end{equation}

Then, for some $\delta >0$:

\begin{eqnarray}
\left| \nabla_{\theta}m_{\theta}(\y_{i})\right|^{2+\delta}&=&\left|\nabla_{\theta}||l_{\theta}||_{2}^2 - 2\varphi(y_{i+1})u^*_{\nabla_{\theta}l_{\theta}}(y_{i}) \right|^{2+\delta}\nonumber\\
&\leq&  C_{1}\left|\nabla_{\theta}||l_{\theta}||_{2}^2\right|^{2+\delta}+C_{2}\left|\varphi(y_{i+1})u^*_{\nabla_{\theta}l_{\theta}}(y_{i}) \right|^{2+\delta},\label{ca}
\end{eqnarray}
where $C_{1}$ and $C_{2}$ are two positive constants. By assumption, the function $||l_{\theta}||_{2}^2$ is twice continuously differentiable w.r.t $\theta$. Hence, $\theta \mapsto \nabla_{\theta}||l_{\theta}||_{2}^2$ is continuous on the compact subset $\Theta$ and the first term of equation (\ref{ca}) is finite. The second term is finite by the moment assumption \textbf{(A9)}.\\

\noindent \paragraph{\small{Hessian Local dominance}:} for $j,k \in\left\{1,\cdots,p\right\}$, $\frac{\p^2 l_{\theta}}{\p\theta_j \p\theta_k} \in \mathbb{L}_{1}(\R)$, the Lebesgue differentiation Theorem gives:

\begin{equation*}
\nabla^{2}_{\theta}m_{\theta}(\y_{i})=\nabla^{2}_{\theta}||l_{\theta}||_{2}^2 - 2\varphi(y_{i+1})u^*_{\nabla^{2}_{\theta}l_{\theta}}(y_{i}),
\end{equation*}
and, for some neighbourhood $\mathcal{U}$ of $\theta_0$:

\begin{equation*}
\E\left[\sup_{\theta \in \mathcal{U} }\left\|\nabla_{\theta}^{2}m_{\theta}(\Y_i)\right\|\right]\leq \sup_{\theta \in \mathcal{U}}\left\|\nabla^{2}_{\theta}||l_{\theta}||_{2}^2\right\|+2\E\left[\sup_{\theta \in \mathcal{U}}\left\|\varphi(Y_{i+1})u^{*}_{\nabla_{\theta}^{2}l_{\theta}}(Y_i)\right\|\right].
\end{equation*} The first term of the above equation is finite by continuity and compactness argument and the second term is finite by the Hessian local dominance assumption \textbf{(A10)}. 

\end{proof}

\subsubsection{Proof of Corollary \ref{lele}}\label{PC}
By replacing $\nabla_{\theta}m_{\theta}(\Y_{1})$ by its expression (\ref{expres}), we have:

\begin{eqnarray*}
\Omega_{0}(\theta)&=&\Var\left[\nabla_{\theta}||l_{\theta}||_{2}^2 - 2\varphi(Y_{2})u^*_{\nabla_{\theta}l_{\theta}}(Y_{1}) \right]\\
&=&4 \Var\left[\varphi(Y_{2})u^*_{\nabla_{\theta}l_{\theta}}(Y_{1}) \right]\\
&=&4\left[\E \left[\varphi(Y_{2})^2 \left(u^*_{\nabla_{\theta}l_{\theta}}(Y_{1}) \right)\left(u^*_{\nabla_{\theta}l_{\theta}}(Y_{1}) \right)'\right]-\E \left[\varphi(Y_{2}) u^*_{\nabla_{\theta}l_{\theta}}(Y_{1})\right]\E \left[\varphi(Y_{2}) u^*_{\nabla_{\theta}l_{\theta}}(Y_{1})\right]'\right].
\end{eqnarray*}
Furthermore, by Eq.(\ref{mod1}) and by independence of the centered noise $(\varepsilon_2)$ and $(\eta_2)$, we have:

\begin{equation*}
\E\left[ \varphi(Y_{2}) u^*_{\nabla_{\theta}l_{\theta}}(Y_{1})\right] = \E\left[ (b^2_{\theta_{0}}+\sigma^2_{\theta_{0}})(X_1)  u^*_{\nabla_{\theta}l_{\theta}}(Y_{1})\right].
\end{equation*}

Using Fubini's Theorem and Eq.(\ref{mod1}) we obtain:

\begin{eqnarray}\label{fub}
\E\left[ (b^2_{\theta_{0}}+\sigma^2_{\theta_{0}})(X_1) u^*_{\nabla_{\theta}l_{\theta}}(Y_{1})\right] &=& \E\left[(b^2_{\theta_{0}}+\sigma^2_{\theta_{0}})(X_1) \int e^{iY_{1}z} u_{\nabla_{\theta}l_{\theta}}(z) dz  \right]\nonumber\\
&=&\E\left[(b^2_{\theta_{0}}+\sigma^2_{\theta_{0}})(X_1) \int \frac{1}{2\pi}\frac{1}{f_{\varepsilon}^*(z)}e^{iY_{1}z} (\nabla_{\theta}l_{\theta})^*(-z)dz \right]\nonumber\\
&=&\frac{1}{2\pi} \int \E\left[(b^2_{\theta_{0}}+\sigma^2_{\theta_{0}})(X_1)e^{i(X_1+\varepsilon_1)z} \right] \frac{1}{f_{\varepsilon}^*(z)} (\nabla_{\theta}l_{\theta})^*(-z) dz \nonumber\\
&=&\frac{1}{2\pi} \int \frac{\E\left[ e^{i\varepsilon_1z}\right]}{f_{\varepsilon}^*(z)} \E\left[(b^2_{\theta_{0}}+\sigma^2_{\theta_{0}})(X_1)e^{iX_1z}\right] (\nabla_{\theta}l_{\theta})^*(-z) dz\nonumber,
\end{eqnarray}
so that 
\begin{eqnarray}\label{fub}
\E\left[ (b^2_{\theta_{0}}+\sigma^2_{\theta_{0}})(X_1) u^*_{\nabla_{\theta}l_{\theta}}(Y_{1})\right] &=&\frac{1}{2\pi} \int \E\left[(b^2_{\theta_{0}}+\sigma^2_{\theta_{0}})(X_1)e^{i(X_1+\varepsilon_1)z} \right] \frac{1}{f_{\varepsilon}^*(z)} (\nabla_{\theta}l_{\theta})^*(-z) dz \nonumber\\
&=&\frac{1}{2\pi} \int \frac{\E\left[ e^{i\varepsilon_1z}\right]}{f_{\varepsilon}^*(z)} \E\left[(b^2_{\theta_{0}}+\sigma^2_{\theta_{0}})(X_1)e^{iX_1z}\right] (\nabla_{\theta}l_{\theta})^*(-z) dz\nonumber\\
&=&\frac{1}{2\pi} \E\left[ (b^2_{\theta_{0}}+\sigma^2_{\theta_{0}})(X_1) \int e^{iX_1z} (\nabla_{\theta}l_{\theta})^*(-z) dz \right]\nonumber\\
&=&\frac{1}{2\pi} \E\left[ (b^2_{\theta_{0}}+\sigma^2_{\theta_{0}})(X_1) \left((\nabla_{\theta}l_{\theta})^*(-X_1)\right)^*\right]\nonumber\\
&=& \E\left[ (b^2_{\theta_{0}}+\sigma^2_{\theta_{0}})(X_1) \nabla_{\theta}l_{\theta}(X_1)\right].
\end{eqnarray}

Hence,
\begin{equation*}
\Omega_{0}(\theta)=4\left(P_{2}-P_{1}\right),
\end{equation*} 
where
\begin{eqnarray*}
&&P_{1}=\E\left[ (b^2_{\theta_{0}}+\sigma^2_{\theta_{0}})(X_1) \nabla_{\theta}l_{\theta}(X_1)\right]\E\left[(b^2_{\theta_{0}}+\sigma^2_{\theta_{0}})(X_1)\nabla_{\theta}l_{\theta}(X_1)\right]',\\
&&P_2=\E \left[\varphi(Y_{2})^2 \left(u^*_{\nabla_{\theta}l_{\theta}}(Y_{1}) \right)\left(u^*_{\nabla_{\theta}l_{\theta}}(Y_{1}) \right)'\right].
\end{eqnarray*}

\emph{Calculus of the covariance matrix of Corollary (\ref{lele}):} By replacing $(\nabla_{\theta}m_{\theta}(Y_{1}))$ by its expression (\ref{expres}) we have:

\begin{eqnarray*}
\Omega_{j-1}(\theta)
&=&\C\left(\nabla_{\theta}||l_{\theta}||_{2}^2 - 2\varphi(Y_{2})u^*_{\nabla_{\theta}l_{\theta}}(Y_{1}) , \nabla_{\theta}||l_{\theta}||_{2}^2 - 2\varphi(Y_{j+1})u^*_{\nabla_{\theta}l_{\theta}}(Y_{j}) \right),\\
&=&4\C\left(\varphi(Y_{2})u^*_{\nabla_{\theta}l_{\theta}}(Y_{1}) , \varphi(Y_{j+1})u^*_{\nabla_{\theta}l_{\theta}}(Y_{j})\right),\\
&=&4\left[\E\left(\varphi(Y_{2})u^*_{\nabla_{\theta}l_{\theta}}(Y_{1})\varphi(Y_{j+1})u^*_{\nabla_{\theta}l_{\theta}}(Y_{j})\right)-\E\left(\varphi(Y_{2})u^*_{\nabla_{\theta}l_{\theta}}(Y_{1})\right)\E\left(\varphi(Y_{j+1})u^*_{\nabla_{\theta}l_{\theta}}(Y_{j})\right)'\right].
\end{eqnarray*}

By using Eq.(\ref{fub}) and the stationary property of the $Y_i$, one can replace the second term of the above equation by: 

\begin{equation*}
\E\left[(b^2_{\theta_{0}}+\sigma^2_{\theta_{0}})(X_1)\nabla_{\theta}l_{\theta}(X_1)\right]\E\left[(b^2_{\theta_{0}}+\sigma^2_{\theta_{0}})(X_1)\nabla_{\theta}l_{\theta}(X_1)\right]'.
\end{equation*}

Furthermore, by using Eq.(\ref{mod1}) we obtain:

\begin{eqnarray}
\E\left[\varphi(Y_{2})\varphi(Y_{j+1})u^{*}_{\nabla_{\theta}l_{\theta}}(Y_{1}) u^{*}_{\nabla_{\theta}l_{\theta}}(Y_{j})\right]
&=&\E\left[(b^2_{\theta_{0}}+\sigma^2_{\theta_{0}})(X_1)(b^2_{\theta_{0}}+\sigma^2_{\theta_{0}})(X_j)u^{*}_{\nabla_{\theta}l_{\theta}} (Y_{1})u^{*}_{\nabla_{\theta}l_{\theta}}(Y_{j})\right]\nonumber\\
&+&\E\left[(b^2_{\theta_{0}}+\sigma^2_{\theta_{0}})(X_1)\left(\eta_{j+1}+\varepsilon_{j+1}\right)u^{*}_{\nabla_{\theta}l_{\theta}}(Y_{1})u^{*}_{\nabla_{\theta}l_{\theta}}(Y_{j})\right]\label{1}\\
&+&\E\left[(b^2_{\theta_{0}}+\sigma^2_{\theta_{0}})(X_j)\left(\eta_2+\varepsilon_2\right)u^{*}_{\nabla_{\theta}l_{\theta}}(Y_{1})u^{*}_{\nabla_{\theta}l_{\theta}}(Y_{j})\right]\label{2}\\
&+&\E\left[\left(\eta_2+\varepsilon_2\right)\left(\eta_{j+1}+\varepsilon_{j+1}\right)u^{*}_{\nabla_{\theta}l_{\theta}}(Y_{1})u^{*}_{\nabla_{\theta}l_{\theta}}(Y_{j})\right]\label{3}.
\end{eqnarray}

$\newline$
By independence of the centered noise, the term (\ref{1}), (\ref{2}) and (\ref{3}) are equal to zero. Now, if we use Fubini's Theorem we have: 

\begin{equation}
\E\left[(b^2_{\theta_{0}}+\sigma^2_{\theta_{0}})(X_1)(b^2_{\theta_{0}}+\sigma^2_{\theta_{0}})(X_j)u^{*}_{\nabla_{\theta}l_{\theta}}(Y_{1})u^{*}_{\nabla_{\theta}l_{\theta}}(Y_{j})\right]=\E\left[(b^2_{\theta_{0}}+\sigma^2_{\theta_{0}})(X_1) (b^2_{\theta_{0}}+\sigma^2_{\theta_{0}})(X_j)\nabla_{\theta}l_{\theta}(X_1) \nabla_{\theta}l_{\theta}(X_j)\right].
\end{equation}

Hence, the covariance matrix is given by:
 
\begin{eqnarray*}
\Omega_{j-1}(\theta)&=&4\Bigg(\E\left[(b^2_{\theta_{0}}+\sigma^2_{\theta_{0}})(X_1)(b^2_{\theta_{0}}+\sigma^2_{\theta_{0}})(X_j)\left(\nabla_{\theta}l_{\theta}(X_1)\right)\left(\nabla_{\theta}l_{\theta}(X_j)\right)'\right]\\
&&\quad -\E\left[(b^2_{\theta_{0}}+\sigma^2_{\theta_{0}})(X_1)\left(\nabla_{\theta}l_{\theta}(X_1)\right)\right]\E\left[(b^2_{\theta_{0}}+\sigma^2_{\theta_{0}})(X_1)\left(\nabla_{\theta}l_{\theta}(X_1)\right)\right]'\Bigg)\\
&=&4\left(\tilde{C}_{j-1}-\E\left[(b^2_{\theta_{0}}+\sigma^2_{\theta_{0}})(X_1)\left(\nabla_{\theta}l_{\theta}(X_1)\right)\right]\E\left[(b^2_{\theta_{0}}+\sigma^2_{\theta_{0}})(X_1)
\left(\nabla_{\theta}l_{\theta}(X_1)\right)\right]'\right)\\
&=&4\left(\tilde{C}_{j-1}-P_{1}\right).
\end{eqnarray*}
Finally, we obtain: $\Omega(\theta)=\Omega_{0}(\theta)+2\sum_{j >1}^{\infty}\Omega_{j-1}(\theta)$ with $\Omega_{0}(\theta)=4\left(P_2-P_{1}\right)$ and $\Omega_{j-1}(\theta)=4\left(\tilde{C}_{j-1}-P_{1}\right)$.

$\newline$
\emph{Expression of the Hessian matrix $V_{\theta}$ :} We have:

\begin{equation}
\mathbf{P}m_{\theta} = ||l_{\theta}||_{2}^2 - 2\left\langle l_{\theta}, l_{\theta_0}\right\rangle. 
\end{equation}

For all $\theta$ in $\Theta$, the application $\theta \mapsto \mathbf{P}m_{\theta}$ is twice differentiable w.r.t $\theta$ on the compact subset $\Theta$. And for $j\in \left\{1,\cdots,p\right\}$:

\begin{eqnarray*}
\frac{\p \mathbf{P}m}{\p \theta_j}(\theta)&=& 2 \left\langle \frac{\p l_{\theta}}{\p \theta_j}, l_{\theta}\right\rangle-2 \left\langle  \frac{\p l_{\theta}}{\p \theta_j}, l_{\theta_0}\right\rangle\\
&=&2 \left\langle \frac{\p l_{\theta}}{\p \theta_j}, l_{\theta}-l_{\theta_0}\right\rangle,\\
&=&0 \text{ at the point } \theta_{0},
\end{eqnarray*}
and for $j,k \in \left\{1,\cdots,p\right\}$:

\begin{eqnarray*}
\frac{\p^2 \mathbf{P}m}{\p\theta_j \p\theta_k}(\theta) &=& 2 \left(\left\langle \frac{\p^2 l_{\theta}}{\p \theta_j \theta_k}, l_{\theta}- l_{\theta_0}\right\rangle+\left\langle \frac{\p l_{\theta}}{\p  \theta_k}, \frac{\p l_{\theta}}{\p \theta_j}\right\rangle\right)_{j,k}\\
&=& 2 \left(\left\langle \frac{\p l_{\theta}}{\p  \theta_k}, \frac{\p l_{\theta}}{\p \theta_j}\right\rangle\right)_{j,k}\text{ at the point } \theta_{0}. \qquad \Box
\end{eqnarray*}

\subsection{M-estimator using the example in Section \ref{resultats}}\label{appendice2}

\paragraph{\small{Expression of $f_{\varepsilon}^{*}$}:}

\noindent consider the random variable $\overline{\varepsilon}=\frac{\varepsilon-C}{\sqrt{V}}$ with $\varepsilon=\log(X^2)$ where $X$ is standard Gaussian random variable, $C=\E[\log(X^2)]$ and $V=\V[\log(X^2)]$. The Fourier transform of $\overline{\varepsilon}$ is given by:
\begin{eqnarray*}
\E\left[\exp\left(i\overline{\varepsilon} y\right)\right]&=&\exp\left(-\frac{iC}{\sqrt{V}}y\right)\E\left[\exp\left(i\varepsilon y\right)\right]\\
&=&\exp\left(-\frac{iC}{\sqrt{V}}y\right)\E\left[X^{\frac{2iy}{\sqrt{V}}}\right]\\
&=&\exp\left(-\frac{iC}{\sqrt{V}}y\right)\frac{1}{\sqrt{2\pi}}\int_{-\infty}^{+\infty} x^{\frac{2iy}{\sqrt{V}}}\exp\left(-\frac{x^2}{2}\right)dx
\end{eqnarray*}
Using a change of  variable $z=\frac{x^2}{2}$, we get:
\begin{eqnarray*}
\E\left[\exp\left(i\overline{\varepsilon} y\right)\right]&=&\exp\left(-\frac{iC}{\sqrt{V}}y\right)\frac{2^{\frac{iy}{\sqrt{V}}}}{\sqrt{\pi}}\int_{0}^{+\infty} z^{\frac{iy}{\sqrt{V}}-\frac{1}{2}}e^{-z}dz\\
& \equiv &\exp\left(-\frac{iC}{\sqrt{V}}y\right)\frac{2^{\frac{iy}{\sqrt{V}}}}{\sqrt{\pi}}\Gamma\left(\frac{1}{2}+\frac{iy}{\sqrt{V}}\right).
\end{eqnarray*}
Then
\begin{eqnarray*}
f_{\overline{\varepsilon}}^{*}&=&\exp\left(-\frac{iC}{\sqrt{V}}y\right)\frac{2^{\frac{iy}{\sqrt{V}}}}{\sqrt{\pi}}\int_{0}^{+\infty} z^{\frac{iy}{\sqrt{V}}-\frac{1}{2}}e^{-z}dz\\
& \equiv &\exp\left(-\frac{iC}{\sqrt{V}}y\right)\frac{2^{\frac{iy}{\sqrt{V}}}}{\sqrt{\pi}}\Gamma\left(\frac{1}{2}+\frac{iy}{\sqrt{V}}\right).
\end{eqnarray*}

\paragraph{\small{The CIR process: }}

\noindent taking that $\eta_{i+1} \sim \mathcal{N}(0,1)$ and $\varepsilon_{i}$ has a (log-) Chi-squared probability density function,  if the Feller's condition holds true ($a=\frac{2\kappa\theta}{\sigma^2}\geq 1$) and $c=\frac{2\kappa}{\sigma^2}>0$, then the volatility process $X_i$ is stationary ergodic and $\rho-mixing$. The stationary distribution $f_{\theta}$ is the gamma distribution $\gamma(a,c)$ (see \cite{GvJtLc99}). On the other hand, the functions $b_{\theta}$, $\sigma_{\theta}$ and $l_{\theta}$ are given by:
\begin{equation*}
b_{\theta}(x)=(1-\kappa)x+\kappa\theta \text{ and } \sigma_{\theta}(x)=\sigma \sqrt{x}
\end{equation*}
\begin{eqnarray*}
\qquad \qquad \qquad \qquad \qquad l_{\theta}(x)&=&\left(b^2_{\theta}(x)+\sigma^2_{\theta}(x)\right)\gamma(a,c) \\
&=&\left(\left(1-\kappa\right)^2 x^{2}+2x\left(1-\kappa\right)\kappa\theta+\sigma^2+(\kappa\theta)^2\right)f_{\theta}(x),\\
&=&\left(\alpha_1x^{2}+\alpha_2 x+\alpha_3\right)f_{\theta}(x).
\end{eqnarray*}
where $\theta=(\kappa, \mu, \sigma)$ and $\alpha_1=\left(1-\kappa\right)^2,\ \alpha_2=2\left(1-\kappa\right)\kappa\theta+\sigma^2$, $\ \alpha_3=(\kappa\theta)^2$. \\

Therefore
\begin{eqnarray*}
l^{*}_{\theta}(t)&=&\E[e^{itX}l_{\theta}(X)]\\
&=&\alpha_1\E[X^2 e^{itX}]+\alpha_2\E[X e^{itX}]+\alpha_3\E[e^{itX}] \text{ with } X \sim \Gamma(a,c)\\
&=&-\alpha_1\frac{\p^2 f^{*}_{\theta}}{\p t^2}(t)-i\alpha_2\frac{\p f^{*}_{\theta}}{\p t}(t)+\alpha_3 f^{*}_{\theta}(t).
\end{eqnarray*}
After replacing $f^{*}_{\theta}(t)$ by $\left(1-\frac{it}{c}\right)^{-a}$, we obtain:
\begin{equation*}
l_{\theta}^{*}(x)=-\alpha_1\left[\frac{-a}{c^{2}}(a+1)\left(1-\frac{ix}{c}\right)^{-a-2}\right]+i\alpha_2\frac{a}{c}\left(1-\frac{ix}{c}\right)^{-a-1}+\alpha_3\left(1-\frac{ix}{c}\right)^{-a}.
\end{equation*}
It follows that for all $\theta$, the squared norme of $l_\theta$ is given by:
\begin{eqnarray*}
\left\|l_{\theta}\right\|^2&=&\int_{}^{}\left(b_{\theta}^{2}+\sigma^{2}_{\theta}(x)\right)^{2}\Gamma^{2}(a,c)dx\\
&=&\int_{}^{}\left(\beta_1x^4+\beta_2x^3+\beta_3x^2+\beta_4x+\beta_5\right)\Gamma^{2}(a,c)dx,\\
\end{eqnarray*}
 where $\beta_1=\alpha_1^2,\ \beta_2=2\alpha_1\alpha_2,\ \beta_3=2\alpha_1\alpha_3+\alpha_2^2,\ \beta_4=2\alpha_2\alpha_3,\ \beta_3=\alpha_3^2$. Finally, using the non-centered moments of a Gamma-distributed random variable, $\E[X^r]=\frac{\Gamma(a+r)}{\Gamma(a)c^{r}}$, we get:
\begin{eqnarray*}
\beta_1\int_{}^{}x^4\Gamma^{2}(a,c)dx&=&\beta_1\int_{}^{}x^4\frac{c^{2a}}{\Gamma^{2}(a)}e^{-2cx}x^{2a-2}\\
&=&\beta_1 2^{-(2a+3)}c^{-3}\frac{\Gamma(2a+3)}{\Gamma^{2}(a)}\int\frac{(2c)^{2a+3}}{\Gamma(2a+3)}e^{-(2c)x}x^{(2a+3)-1}dx\\
&=&\beta_1 2^{-(2a+3)}c^{-3}\frac{\Gamma(2a+3)}{\Gamma^{2}(a)}\int\Gamma(2a+3,2c)dx\\
&=&\beta_1 2^{-(2a+3)}c^{-3}\frac{\Gamma(2a+3)}{\Gamma^{2}(a)}
\end{eqnarray*}
and
\begin{eqnarray*}
&&\beta_2\int_{}^{}x^3\Gamma^{2}(a,c)dx = \beta_2 2^{-(2a+2)}c^{-2}\frac{\Gamma(2a+2)}{\Gamma^{2}(a)}\int\Gamma(2a+2,2c)dx =\beta_2 2^{-(2a+2)}c^{-2}\frac{\Gamma(2a+2)}{\Gamma^{2}(a)},\\
&&\beta_3\int_{}^{}x^2\Gamma^{2}(a,c)dx = \beta_3 2^{-(2a+1)}c^{-1}\frac{\Gamma(2a+1)}{\Gamma^{2}(a)}\int\Gamma(2a+1,2c)dx =\beta_3 2^{-(2a+1)}c^{-1}\frac{\Gamma(2a+1)}{\Gamma^{2}(a)},\\
&&\beta_4\int_{}^{}x\Gamma^{2}(a,c)dx =\beta_4 2^{-(2a)}\frac{\Gamma(2a)}{\Gamma^{2}(a)},\\
&&\beta_5\int_{}^{}\Gamma^{2}(a,c)dx=\beta_5 2^{-2a+1}c\frac{\Gamma(2a-1)}{\Gamma^{2}(a)}.
\end{eqnarray*}
and the expression of the contrast function (\ref{contraste_application3}) is obtained. It is worth noting that the function $u^{*}_{l_{\theta}}(y)$ defined in Definition \ref{contrast_function} must be approximated numerically by using standard quadrature methods or Fast Fourier Transform.

\subsubsection{Proof of Theorems \ref{consistency_theorem} and \ref{CLT} for the CIR process}\qquad\\

\textbf{Mixing property:} under the Feller's condition, the volatility process $X_t$ is $\rho$-mixing and so $\alpha$-mixing by using the strong Markov property. \\

\textbf{Regularity conditions:} for the CIR process, the function $l_{\theta}$ is given by the following polynomial function $x\mapsto (\alpha_1x^2+\alpha_2x +\alpha_3)f_{\theta}(x)$ with $\alpha_1=\left(1-\kappa\right)^2, \alpha_2=2\left(1-\kappa\right)\kappa\theta+\sigma^2$, $\ \alpha_3=(\kappa\theta)^2$ and for all $x$ this function is smooth w.r.t. $\theta \in \Theta$. Hence, it remains to prove the moment condition and the local dominance to apply Theorem \ref{CLT}.\\

Since the function $l_{\theta}$ is polynomial w.r.t $\theta$ belonging to the compact subset $\Theta$, all the derivatives exist and in particular $\sup_{\theta \in \Theta}l_{\theta}$ and $\sup_{\theta \in \Theta}\nabla^{2}_{\theta}l_{\theta}$ are finite. Furthermore, by combining the compactness argument and as the Fourier transform $f_{\epsilon}^{*}$ satisfies (see \cite{Fan}):

\begin{equation*}
 |f^{*}_{\varepsilon}(x)|=\sqrt{2}\exp\left(-\frac{\pi}{2} |x|\right)\left(1+O\left(\frac{1}{|x|}\right)\right), \quad |x|\rightarrow \infty,             
\end{equation*}
which means that $f_{\varepsilon}$ is ordinary-smooth in its terminology, we obtain:

\begin{equation*}
\left\lbrace\begin{array}{ll}
\E\left(\sup_{\theta \in \Theta}\left\|\varphi(Y_{2})u^{*}_{l_{\theta}}(Y_{1})\right\|\right)<\infty \\
\E\left(\left|\varphi(Y_{2})u^*_{\nabla_{\theta}l_{\theta}}(Y_{1}) \right|^{2+\delta}\right)<\infty \qquad\qquad \text{ for some } \delta >0,\\
\E\left(\sup_{\theta \in \mathcal{U}}\left\|\varphi(Y_{2})u^{*}_{\nabla_{\theta}^{2}l_{\theta}}(Y_{1})\right\|\right)<\infty \qquad \text{ for some neighbourhood } \mathcal{U} \text{ of } \theta_0 .
\end{array}
\right.
\end{equation*}.

\end{document}